\documentclass[11pt]{amsart}
\usepackage[centering]{geometry} % See geometry.pdf to learn the layout options. There are lots.
\usepackage{hyperref}
\usepackage{graphicx}
\usepackage{mathrsfs}
\usepackage{amssymb}
\usepackage{epstopdf}
\usepackage{amscd}
\usepackage[english]{babel}

\newtheorem{theorem}{Theorem}

\newtheorem{conj}{Conjecture}

\newtheorem{proposition}[theorem]{Proposition}
\newtheorem{corollary}[theorem]{Corollary}
\newtheorem{claim}[theorem]{Claim}
\newtheorem{lemma}[theorem]{Lemma}
\newtheorem{remark}[theorem]{Remark}
\newtheorem{example}[theorem]{Example}

\newtheorem{question}[theorem]{Question}

\newtheorem{fact}[theorem]{Fact}
\newtheorem*{theorem-non}{Theorem}

\newtheorem*{factt}{Fact}

\DeclareMathOperator{\cha}{char}
\DeclareMathOperator{\Sym}{Sym}
\DeclareMathOperator{\norm}{norm}

\DeclareMathOperator{\Bl}{Bl}
\DeclareMathOperator{\Pic}{Pic}

\DeclareMathOperator{\Ker}{Ker}

\DeclareMathOperator{\Supp}{Supp}

\DeclareMathOperator{\rk}{rank}
\DeclareMathOperator{\Amp}{Amp}
\DeclareMathOperator{\ord}{ord}
\DeclareMathOperator{\Jac}{Jac}

\author{Federico Buonerba}
\address{\tiny{Federico Buonerba\newline Courant Institute of Mathematical Sciences,
 New York University, 
 251 Mercer Street, 
 New York, NY 10012, USA}}
 \email{buonerba@cims.nyu.edu}

\author{Fedor Bogomolov} 
\address {\tiny{Fedor Bogomolov \newline Courant Institute of Mathematical Sciences,
 New York University, 
 251 Mercer Street, 
 New York, NY 10012, USA\newline
 National Research University Higher School of Economics, Russian Federation,
AG Laboratory, HSE, 7 Vavilova str., Moscow, Russia, 117312}}
\email{bogomolo@cims.nyu.edu}

\title{Dominant classes of projective varieties}
\begin{document}
\begin{abstract}

We give evidence for a uniformization-type conjecture, that any algebraic variety can be altered
into a variety endowed with a tower of smooth
fibrations of relative dimension one. 

\end{abstract}

\maketitle

%\section{}
%\subsection{}
\markboth{F.Buonerba \& F.Bogomolov}{Dominant classes of projective varieties}

The problem of constructing a resolution of singularities of projective varieties is one of the
most fundamental obstructions to our understanding of their analytic, arithmetic and geometric
properties. A tremendous amount of work has shed light on the this problem, yet in its full generality
it is still wide open over fields of positive characteristic.
A cornerstone result, albeit conjecturally not optimal, is de Jong's \cite{2} 3.1, stating that any variety
can be altered into a smooth projective one. Allowing alterations, other than birational modifications,
comes along with a profusion of new natural questions, in the spirit of: which further properties
can we require, for a class of smooth projective varieties, to be dominant? 
Recall from \cite{1} that a class $\mathscr C(k,n)$ of $n$-dimensional projective varieties over a field $k$
is dominant if for every projective $n$-dimensional variety $X$, there exists $Y\in \mathscr C(k,n)$ and a surjective
$k$-morphism $Y\to X$. In this terminology, de Jong's
result says that smooth projective varieties form a dominant class over any
field and in any dimension. Constructing minimal classes of dominant varieties is a problem that attracted
attention, and a satisfactory answer is still unknown even in the case of curves over fields
that are finitely generated over their prime subfield. Some results, questions and speculations in this direction
can be found in \cite {4},\cite{8},\cite{5}.
In this paper we give evidence for the following conjecture:
\begin{conj}[\cite{1}]\label{conj}
For any field $k$ and positive integer $n$, the class of $n$-dimensional smooth projective varieties
$X$, endowed with a tower of smooth fibrations $$X\to X_1\to...\to X_n$$ with $\dim X_i=n-i$, is dominant.
\end{conj}
Section \ref{one} will be devoted to the proof of:
\begin{theorem}\label{thmmm}
The following classes $\mathscr C(k,n)$ are dominant:
\begin{itemize}
\item[(i)] For $n=3$, smooth threefolds with a smooth connected morphism onto a smooth curve.  
\item[(ii)] For any $n$ and $k$ a finite field, projective varieties admitting a connected morphism
onto a smooth curve, with only one singular fiber whose singular locus consists of one ordinary double point.
\end{itemize}
\end{theorem}
Let us give a quick indication of the proof. The idea is to construct fibrations
using Lefschetz pencils. In fact, the existence of Lefschetz pencils on smooth
projective varieties, \cite{3} XVII, combined with de Jong's alteration result, \cite{2} 3.1, immediately gives:
\begin{factt} 
For any field $k$ and integer $n\geq 2$, the class of projective varieties admitting a connected morphism
onto $\mathbf P^1_k$, with isolated singular fibers whose singular locus consists of one ordinary double point, is dominant.
\end{factt}
Statement (i) of the Theorem
is then an immediate consequence of the Brieskorn-Tyurina's simultaneous resolution of surface ordinary double points,
which in fact provides a simultaneous resolution of the fibers of the fibration induced by the Lefschetz pencil.\\
Statement (ii) is more delicate. We are given $X$ a smooth and projective
over a finite field, with $ X^{\vee}$ the dual variety of singular hyperplane sections. We
construct a curve $C$, in the space of hyperplanes of $X$, that intersects $X^{\vee}$ in a single point,
which is a smooth point for both $X^{\vee}$ and $C$.  In order to perform this construction, the
crucial assumption, that $k$ is a finite field, manifests itself in that the Picard group - of degree zero
divisor classes - is always finite for projective curves.
The total space of the induced family of hyperplane sections has a natural fibration onto $C$ with the required
conditions on the fibers, and by construction it
has a surjective morphism onto $X$.\\
In section \ref{two} we focus our attention on surfaces. In this case the dual variety $X^{\vee}$
stratifies according to geometric genus of the generic member, and therefore one might try to
understand the geometry and the modularity of such strata. The general idea is outlined in
Proposition \ref{strategy}.
The geometric structure of the stratification provides satisfactory answers for surfaces of negative Kodaira dimension,
on which complete families of curves with constant geometric genus are easily constructed in Proposition \ref{negkod}. 
Such families induce, upon normalization of the total space, an equigeneric fibration with smooth general member.
Unfortunately the situation becomes complicated in non-negative Kodaira dimensions, where the
method doesn't provide any obvious answer on general K3's and hypersurfaces in $\mathbf P^3$
of degree at least $4$. The difficulty here is that the strata might be nested into each
other as divisors, and it seems hard to insert a complete curve between two consecutive ones.
Therefore we turn our attention to the modularity of the stratification. Less vaguely, each stratum
$S^k$, the closure of the set of curves with geometric genus $g^k$, admits a rational map
 $S^k\dashrightarrow \overline{\mathscr M}_{g^k}$, and one might try to lift complete curves
from $\mathscr M_{g^k}$ to $S^k$.  
Question \ref{quest} summarizes the difficulty with this approach, due essentially to 
the non-genericity of the image of this rational map, whose behaviour near the Deligne-Mumford
boundary is, a priori, arbitrary.\\
In section \ref{three} we continue our discussion by pointing out that the category of surfaces that can be 
dominated by complete families of smooth curves, which conjecturally is everything,
is in fact rather flexible. By this we mean that it is somewhat natural and easy to
create new smooth fibrations out of old ones,
using ideas inspired by Kodaira's construction, \cite{7}, of non-isotrivial smooth fibrations.
First we show in Proposition \ref{catanese}, that any
product of two curves can be dominated by a non-isotrivial smooth fibration. Finally we prove,
in Proposition \ref{fibered},
that any two
smooth fibrations can be dominated simultaneously by a third one. \\
In section \ref{four}, we conclude our discussion by analyzing several aspects of surfaces of general
type that carry an everywhere smooth foliation.
It turns out, Proposition \ref{foliation}, that such a surface must have positive topological index, 
which leads one to think about Kodaira fibrations. Indeed Brunella, in \cite{9},
has set the the foundations of the uniformization theory of such foliated surfaces, by 
establishing that the universal cover has the structure of a disk bundle over a disk.
This, together with Corlette-Simpson's classification, \cite{14}, of Zariski dense Kahler representations
in $PSL_2(\mathbf R)$ leads us to Theorem \ref{BCS}: a smoothly foliated surface of general type
is either a Kodaira fibration, or a foliated subvariety of a polydisk quotient. It is extremely reasonable
that, in the second case, our initial surface is itself a bidisk quotient.\\
We remark that the existence of ball and bidisk quotients can be used to prove, as in Proposition \ref{balls}, that
surfaces with an \'etale cover which is a Stein submanifold in a $3$-dimensional ball form
a dominant class.\\
Switching our attention to smooth foliations on surfaces over fields of characteristic $p>0$,
the situation becomes as different from the complex case, as pleasant. Indeed we have:
\begin{theorem}\label{charpfoliations}
 Let $k$ be a field with $p:=\cha k>2$, and let $X/k$ be an algebraic surface.
 Then there exists a birational modification of $X$, followed by an inseparable cyclic cover, such that
 the resulting surface $Y$ carries a smooth
 $p$-closed foliation.
\end{theorem}
It is worth remarking that the construction of $Y$ is extremely generic, in that we start with a
general Lefschetz 
pencil in $X$, pick a general curve going through the nodes in the pencil, and
finally take an inseparable cyclic cover branched along such curve. The foliation defining
the Lefschetz pencil is shown to pull back, upon saturation, to a smooth and $p$-closed
foliation, via a trivial local computation.\\
Due to the lack of Brunella's theorem in positive characteristic,
the uniformization-type consequences that can be deduced from this
statement, if any, are completely mysterious.

\vspace{0.5cm}

\noindent
\textbf{Acknowledgements} 
We are deeply grateful to Michael McQuillan for his interest in this work,
and for fundamental criticism of an attempt to prove the main conjecture over finite fields.
We would also like to thank Misha Gromov, for several insights and pleasant conversations.
The second author acknowledges that the article
was prepared within the framework of a subsidy granted to the HSE by
the Government of the Russian Federation for the implementation of
the Global Competitiveness Program. The second author was partially
supported by EPSRC programme grant EP/M024830, Simons Fellowship
and Simons travel grant.\\

\section{Proof of theorem 1}
\label{one}
Let $X$ be a smooth projective variety of dimension $n$
over a field $k$, and let $\mathbf P_X$ denote the projective space
 of hyperplane sections of $X$, with universal family 
 $$\begin{CD}
\mathscr U_X@>>> X\\
@VuVV\\
\mathbf P_X
\end{CD}$$\\
Let us prove (i). Assume $n=3$, and consider a Lefschetz 
pencil $f:\mathbf P^1_k\to \mathbf P_X$. 
There is a Zariski-closed subset $S\subset \mathbf P^1_k$ such that $(f^*\mathscr U_X)_s$ has a single ordinary 
double point if and only if $s\in S$. By the Brieskorn-Tyurina's simultaneous
resolution of ordinary double points of surfaces, \cite{6}, there is a ramified cover $C\to \mathbf P^1_k$, and a birational morphism 
$Z\to C\times_{\mathbf P^1_k} f^*\mathscr U_X$ such that the composite $Z\to C$ is a smooth morphism.\qed\\
What follows is the proof of (ii). Assume $k$ is a finite field, 
and denote by $X^{\vee}\subset \mathbf P_X$ the dual variety of singular hyperplanes. It is well known, \cite{3} XVII,
that upon replacing the projective embedding of $X$ with a multiple, $X^{\vee}$ is an
irreducible divisor inside $\mathbf P_X$, whose smooth locus corresponds to hyperplane
sections with a unique singular point, which is an ordinary double point.

\begin{claim}
In order to conclude the proof 
of the theorem, it is enough to find an irreducible curve $C\subset \mathbf P_X$ such that the intersection
$C\cap X^{\vee}$ is supported in a single point, which is smooth for both $C$ and $X^{\vee}$.
\end{claim}
\begin{proof}
Let $f:C\to \mathbf P_X$ be such curve, 
and let $u_C: f^*\mathscr U_X\to C$. $u_C$ has a unique
singular fiber, whose singular locus is a single ordinary double point, 
lying over a smooth point of $C$. Therefore the induced fibration 
$f^*\mathscr U_X\times_C C^{\norm}\to C^{\norm}$
is the required one.
\end{proof}
The rest of the proof will be devoted to the construction of such curve $C$.
Let $S\xrightarrow{\sim} \mathbf P^2_k$
be a general linear plane inside $\mathbf P_X$, intersecting the divisor $X^{\vee}$
along an irreducible, reduced curve $D$. 
Denote by $d=\deg_S(D)$. Recall the well known
\begin{fact}\label{ff}
Since $k$ is finite, the group $\Pic^0_{D/k}$ of degree zero divisor classes on $D$ is finite.
\end{fact}
Let $N>d$ be an integer that kills
$\Pic^0_{D/k}$. Denote by $H$ an hyperplane section of $S$
and by $x$ a smooth point of $D$. 
By Fact \ref{ff} there exists an isomorphism of sheaves
 \[ \mathscr O_{D}(NH) \xrightarrow{\sim} \mathscr O_D(Mx)\]
where of course $M=Nd$, inducing a commutative diagram
$$\begin{CD}
0 @>>> \mathscr O_S(NH-D)@>>> \mathscr O_S(NH) @>>> \mathscr O_D(NH) @>>> 0\\
@. @VVV @| @VVV\\
0 @>>> \Ker \gamma @>>> \mathscr O_S(NH) @>\gamma>> \mathscr O_D(Mx) @>>> 0
\end{CD}$$\\
whose vertical arrows are all isomorphisms.
By our choice of $N$ we deduce that $\Ker \gamma$ has non-trivial global sections, and the points
of $\mathbf P(H^0(S,\Ker \gamma))$ correspond to
curves of degree $N$ inside $S$, whose intersection with $D$ is supported on $x$.
All is left to do is to check that the generic member of the linear system
$\mathbf P(H^0(S,\Ker \gamma))$ is smooth at $x$. This is achieved by way of:

\begin{lemma}
Let $G,L\in k[X,Y,Z]$ be homogeneous polynomials of degrees $d$ and $N>d$ respectively.
Assume that they define irreducible curves intersecting only at $x=[0:0:1]$ and that 
the curve defined by $G$ is smooth at $x$. If 
the curve defined by $L$ is singular at $x$, then $F=Z^{N-d}G+L$ satisfies
\begin{itemize}
\item[(i)] the curves defined by $F$ and $G$ meet only in $x$, and
\item[(ii)] the curve defined by $F$ is smooth at $x$.
\end{itemize}
\end{lemma}
\begin{proof}
$\{x\}=\Supp((G=0)\cap (L=0))=\Supp((G=0)\cap (F=0))$, which is (i).
In order to prove (ii), observe that $x$ is the origin of the affine space $Z=1$,
so then denoting by $u=X/Z$, $v=Y/Z$ and $g(u,v)=Z^{-d}G(X,Y,Z)$, $l(u,v)=Z^{-N}L(X,Y,Z)$ we 
see that $(F=0)$ is defined, around $x$, by the vanishing
of $f=g+l$. Since $\nabla f(0,0)=\nabla g(0,0)\neq 0$ the proof is complete.
\end{proof}
Consequently the generic member of $\mathbf P(H^0(S,\Ker \gamma))$ is smooth at $x$,
and can be taken to be the curve $C$ we are looking for.\qed

\section{An approach to the conjecture for surfaces}
\label{two}
Let us describe some simple examples of surfaces for
which Conjecture \ref{conj} holds:
\begin{example}
 Minimal models of surfaces of negative Kodaira dimension:
 apart from $\mathbf P^2$, these are $\mathbf P^1$-bundles over
 smooth curves.
\end{example}
\begin{example}
 Some surfaces of Kodaira dimension $0$:
 \begin{itemize}
  \item Abelian varieties: Let $A$ be a $d$-dimensional abelian variety,
  and $f:C\to A$ any non-constant algebraic curve. The sum morphism
  $C^d\to A$ given by $(c_1,..., c_d)\to f(c_1)+...+f(c_d)$ is surjective.
  
  \item Kummer K3 surfaces: Let $C$ be a genus $2$ curve with
  hyperelliptic involution $\iota$. Consider the sequence of
  morphisms $$C\times C\to \Sym^2 C\to \Jac^2(C),\ (c_1,c_2)\to c_1+c_2\to [c_1+c_2]$$
  The graph $\Gamma_{\iota}$ of the hyperelliptic involution is
  projected onto a rational curve by the first morphism, and it is
  contracted by the second. Pulling back the above diagram under a
  degree $16$ cover $m_2:\Jac^2(C)\to \Jac^2(C)$ - obtained by identifying
  $\Jac^2 (C) \xrightarrow{\sim} \Pic(C)$ via $K_C$ - we obtain morphisms 
  $$m_2^*(C\times C)\to m_2^*\Sym^2 C\to m_2^*\Sym^2 C/(-1)=K3(C)$$ 
  Since $m_2$ is \'etale, $m_2^*(C\times C)$ is still a product of curves.
 \end{itemize}

\end{example}

The previous example generalizes as follows:
\begin{proposition}
 Let $k$ be a finite field, and $S_0\subset \mathbf P^2(k)$ any finite set.
 Then there exist curves $C_1,C_2$ and a morphism $C_1\times C_2\to \Bl_{\mathbf P^2}(S_0)$.
 
\end{proposition}

\begin{proof}
 Observe that, for any such $S_0$, there exists a finite set $S\subset \mathbf P^1(k)$
 and a morphism $\Bl_{\mathbf P^1\times \mathbf P^1}(S\times S)\to \Bl_{\mathbf P^2}(S_0)$.
 Let $E_1$ be an elliptic curve, and $C$ a curve of genus $2$ with a non-constant
 morphism $C\to E_1$. We can assume wlog that the Jacobian of $C$ is isomorphic
 to $E_1\times E_2$, for some elliptic curve $E_2$. 
 Consider non-constant morphisms $g_i:E_i\to \mathbf P^1$ and
 the resulting 
 $$(g_1,g_2):\Bl_{E_1\times E_2}(g_1^*S\times g_2^*S)\to \Bl_{\mathbf P^1\times \mathbf P^1}(S\times S)$$
 By our assumption on $k$, any finite set of points in $E_1\times E_2$ is torsion, hence
 there exists an integer $N$ and a morphism 
 $$m_N^*\Bl_{E_1\times E_2}(0\times 0)\to \Bl_{E_1\times E_2}(g_1^*S\times g_2^*S)$$
 where $m_N$ is the multiplication by $N$ map on $E_1\times E_2$.
 Finally, composing $$C\times C\to \Sym^2 C\xrightarrow{\sim} \Bl_{E_1\times E_2}(0\times 0)$$
 with the morphisms constructed above, yields 
 $$m_N^*(C\times C)\to m_N^*\Bl_{E_1\times E_2}(0\times 0)\to \Bl_{E_1\times E_2}(g_1^*S\times g_2^*S)
 \to \Bl_{\mathbf P^1\times \mathbf P^1}(S\times S)$$
 In order to conclude, we employ \cite{21} X.1.7 to find curves $C_1, C_2$ and an isomorphism
 $C_1\times C_2\xrightarrow{\sim} m_N^*(C\times C)$.

\end{proof}
Since every Hirzebruch surface is dominated by a blow-up, in
a finite set of points, of $\mathbf P^1\times \mathbf P^1$, we obtain:

\begin{corollary}
 Let $X$ be a Hirzebruch surface over a finite field, and $S\subset X(k)$ any finite set
 of points, then there exist curves $C_1,C_2$ and a morphism $C_1\times C_2\to \Bl_X (S)$.
\end{corollary}

Turning to a more general discussion, let $k$ be a field, $X/k$ a smooth projective surface, and $H$ an ample divisor
such that $H^1(X,\mathscr O_X(nH))=H^2(X,\mathscr O_X(nH))=0$ for all $n\geq 1$.
We have the linear system $\mathbf P_{X,n}=\mathbf P(H^0(X,\mathscr O_X(nH)))$
of dimension  $$d_n=\dim (\mathbf P_{X,n})=\dim\ H^0(X,\mathscr O_X(nH))-1=nH\cdot (nH-K_X)/2+\chi (\mathscr O_X)-1$$
and its generic member 
is smooth of genus $g_n=1+nH\cdot (nH+K_X)/2$.
Observe that $\mathbf P_{X,n}$ admits a natural stratification 
$$\emptyset=:S_n^{N(n)+1}\subsetneq S_n^{N(n)}\subsetneq ...\subsetneq S_n^0:=\mathbf P_{X,n}$$
by, not necessarily irreducible, closed subvarieties, such that the generic member of
each irreducible component of $S_n^k$, $0\leq k\leq N(n)$, corresponds to a
curve with geometric genus at most
$g_n^k$, and of course $0\leq g_n^{N(n)}<...<g_n^0=g_n$. We will call it the
\textit{Severi stratification}.\\
For $g_n^k\geq 2$
there is rational map
$$p:S_n^k\dashrightarrow \mathscr {\overline M}_{g_n^k}$$
into the Deligne-Mumford compactification of the moduli stack of curves.
The reason to introduce
the Severi stratification is:

\begin{proposition}\label{strategy}
\begin{itemize}
\item[]
\item[(i)] Assume that, for some $0\leq k\leq N(n)$ such that $g_k\geq 2$, there exists a smooth proper curve $C$
and a non-constant morphism $f:C\to S_n^k$ such that the rational map $p\circ f:C\dashrightarrow \mathscr {\overline M}_{g_n^k} $
extends, after a finite cover $C'/C$, to a morphism $C'\to \mathscr M_{g_k}$.
Then, upon replacing $C$ by a finite cover, there exists a smooth surface $Y$,
a smooth fibration $Y\to C$, and a surjective morphism $Y\to X$.
\item[(ii)] For some $0\leq k\leq N(n)$, there exists a smooth proper curve $C$
and a non-constant morphism $f:C\to S_n^k\setminus S_n^{k+1}$.
Then, upon replacing $C$ by a finite cover, there exists a smooth surface $Y$,
a fibration $Y\to C$ whose fibers have constant geometric genus, and a surjective morphism $Y\to X$.
\end{itemize}

\end{proposition}
\begin{proof}
\begin{itemize}
\item[]
\item[(i)]As in the previous section, $\mathbf P_{X,n}$ is equipped with a universal space $u:\mathscr U_{X,n}\to \mathbf P_{X,n}$,
and we can consider the fibration $u_C:f^*\mathscr U_{X,n}\to C$. 
The normalization $u_C^{\norm}:(f^*\mathscr U_{X,n})^{\norm}\to C$ 
induces the moduli morphism $p\circ f:C\dashrightarrow \overline{\mathscr M}_{g_k}$
and therefore after replacing $C$ by $C'$, the resulting
fibration is smooth.
\item[(ii)] Since the curves in our family $u_C$ are equigeneric -
meaning that the geometric genus is constant along the fibers -
we have that $u_C^{\norm}$ is again an equigeneric family, with smooth general member.
\end{itemize}
\end{proof}
The natural problem becomes to investigate to what extent Proposition \ref{strategy} can be applied.
The next proposition shows how the geometry of the moduli $\mathscr M_3$ can be
used to construct complete families of smooth curves mapping to the plane, using \ref{strategy}.(i):
\begin{factt}
There exists a complete family $p:Y\to C$ of smooth genus $3$ curves, whose fibrewise
canonical map $|p_*\omega_{Y/C}|:Y\to C\times \mathbf P_k^2$ 
realizes the generic curve as a smooth quartic, and the special ones as double conics with
smooth support.
\end{factt}
\begin{proof}
The Torelli map $\mathscr M_3\to \mathscr A_3$ is bijective on closed points, so there is a canonical
Baily-Borel compactification $\mathscr M^{BB}_3$,
which is projective and with boundary in codimension $2$. Therefore, a generic complete intersection
curve $C\subset \mathscr M^{BB}_3$ is contained in $\mathscr M_3$ and intersects the hyperelliptic locus
$\mathscr H_3\subset \mathscr M_3$ transversely. 
Moreover the canonical map $|\omega_C|:C\to \mathbf P^2_k$ realizes points of $\mathscr H_3$
as plane double conics, and points of $\mathscr M_3\setminus \mathscr H_3$ as smooth plane quartics.
\end{proof}
We refer to \cite{20}, pp.133 for a discussion around the modular behavior of families
plane quartics degenerating to double concics.\\
Something can be said on surfaces with negative Kodaira dimension,
indeed \ref{strategy}.(ii) quickly proves:

\begin{proposition}\label{negkod}
Let $X$ be a smooth surface with negative Kodaira dimension.
Then for $n$ sufficiently big, there exists a smooth proper curve $C$
and a non-constant morphism $f:C\to S_n^k\setminus S_n^{k+1}$. Therefore,
$X$ can be dominated by an equigeneric family of curves.
\end{proposition}
\begin{proof}
Consider the following inductive construction: set $S_0:=S_n^0$, and assuming 
defined $S_k$, let $S_{k+1}\subset S_k$ be an irreducible component of $S_n^{k+1}$ of
maximal dimension. We claim that, for some $k$, $\dim S_k -\dim S_{k+1}\geq 2$:
by assumption the canonical bundle of $X$ is not pseudoefective, therefore
$K_X\cdot H<0$, so then $d_n-g_n=-nH\cdot K_X -2+\chi(\mathscr O_X)$ is positive for $n$
sufficiently big. Hence it is impossible to have $\dim S_k= \dim S_0 -k = d_n-k$ for every $k$.
\end{proof}

\begin{remark}
Unfortunately, this is still not enough to prove the conjecture in negative
Kodaira dimension: obviously, there exist equigeneric families of curves
with smooth generic member, yet carrying singular, necessarily not irreducible, members.
\end{remark}

The situation becomes more interesting in non-negative Kodaira dimension,
where the genus $g_n$ tends to be bigger than the dimension $d_n$, and the
stratification might consist of strata of consecutive codimension one. In
fact this happens on general hypersurfaces:
\begin{theorem}\cite{11}
Let $X$ be a general hypersurface of $\mathbf P^3$ of degree $d\geq 4$. For $n\geq d$
and any $0\leq k\leq d_n$, the variety $S_n^k$ contains an irreducible component of
dimension $d_n-k$ whose generic point parametrizes curves with $k$ nodes.
\end{theorem}

In this situation, the best we can hope in order to construct a 
smooth family is a positve answer to:

\begin{question}\label{quest}
Can we find
$n,k$ such that the boundary divisor $\Delta_{g_n^k}\subset\mathscr {\overline M}_{g_n^k} $, restricted to the closure of
$p(S_{n}^{k})$, is not ample ? Even better, admitting a contraction ?
\end{question}
A positive answer to the above would provide us with
a curve to which apply Proposition \ref{strategy} (i), and hence prove the conjecture in dimension $2$.
The problem is that the image of $p$ in the moduli $\mathscr M_{g_n^k}$ is going
to be of high codimension and extremely non-generic. For example, consider the natural rational map
$$p:\mathscr {\overline M}_{g_n}\dashrightarrow \mathscr M_{g_n}^{BB}$$
Then, unless $X$ is dominated by an isotrivial surface, $\overline {p(S_n^k)}\cap \Delta_{g_n^k}$ is not contracted by $p$, 
albeit $p$ does restrict to a contraction along
$\Delta_{g_n^k}$.
\begin{section}{flexibility of Kodaira fibrations}
\label{three}
In this section we emphasize that the class of 
Kodaira fibrations, i.e. those surfaces admitting
a non-isotrivial smooth morphism onto a smooth curve,
is remarkably flexible, and there plenty of smooth fibrations that
can be constructed out of given ones.
First, let us review Kodaira's original construction, \cite{7}:
given any curve $C_0$ of genus at least $2$, let $C\to C_0$ be
any non-trivial, finite \'etale cover with Galois group $\Gamma$.
Consider, for any $m||\Gamma|$, the natural quotient $\pi_1(C)\to H_1(C,\mathbf Z/m\mathbf Z)$,
and the corresponding \'etale cover $f:C'\to C$ with Galois group $H_1(C,\mathbf Z/m\mathbf Z)$.
The crucial observation is that, by the Kunneth formula,
the class of the graph $\Gamma_f$ inside $H^2(C'\times C, \mathbf Z/m\mathbf Z)$
depends uniquely on the morphism $f^*:H^q(C,\mathbf Z/m\mathbf Z)\to H^q(C',\mathbf Z/m\mathbf Z)$.
By construction, this morphism is trivial when $q=1,2$, while it is an isomorphism when $q=0$.
In particular, the cohomology class of $\Gamma_{\gamma\circ f}, \gamma\in \Gamma$
is independent of $\gamma$. Since $m||\Gamma|$, we deduce that
$D:=\cup_{\gamma} \Gamma_{\gamma\circ f}$ is $m$-divisible in $H^2(C'\times C,\mathbf Z)$.
Let $X=X(C,m)\to C'\times C$ be the cyclic covering of order $m$, branched along $D$.
Since the $\Gamma_{\gamma\circ f}$ are pairwise disjoint and each of them is an \'etale
multisection of the second projection $p_2:C'\times C\to C$, we deduce that
the composition $X\to C$ is a non-isotrivial smooth fibration.\\
We now employ Kodaira's construction as follows:
\begin{proposition}\label{catanese}
Given two curves $C_1,C_2$, there exists a curve $C$, a smooth non-isotrivial fibration in smooth 
curves $Y\to C$ and a finite morphism $Y\to C_1\times C_2$.
\end{proposition}
\begin{proof}
Let $C_0$ be a curve of genus at least $2$ with two surjective morphisms $f_i:C\to C_i$.
Running the above construction, we obtain a Kodaira fibration $X$, with a natural
sequence of finite morphism $X\to C'\times C\to C\times C\to C_1\times C_2$.
\end{proof}
The class of surfaces of general type that are finite quotients of products
of curves is vast.
\cite{12}, and references therein, point to a detailed study of the class
of so-called product-quotient surfaces, which are, by the previous proposition,
also dominated by non-isotrivial smooth fibrations.\\
Similarly to what has been done in Proposition \ref{catanese}, let
$C$ of genus $g\geq 2$, and consider the divisor $\Delta_m\subset \Sym^mC$
inside the $m$-fold symmetric product of $C$, of points $(c_1,...,c_m)$ with $c_i=c_j$ for
some $i\neq j$. 
\begin{proposition}
Let $f:D\to \Sym^mC\setminus \Delta_m$ be a non-constant morphism from a complete
curve $D$. Then there exists a smooth non-isotrivial fibration
 $Y\to D$ and a finite morphism $Y\to C\times D$.
 \end{proposition}
 \begin{proof}
The morphism $f$ defines a divisor $D_f\subset C\times D$, whose fiber over $d\in D$ is the $m$-tuple
of points $f(d)\subset C$, and by construction the projection $p:D_f\to D$ is \'etale.
As before, upon replacing $D$ by a non-trivial Galois cover, one constructs 
Kodaira fibrations as cyclic covers of $D_f\times D$, branched along the Galois orbit
of the graph of $p$.
\end{proof}
Observe that this construction provides many examples when $m=2$, since $\Delta\subset \Sym^2C$
can be contracted via $\Sym^2C\to \Jac^0(C)/(-1)$, $(c_1,c_2)\to [c_1-c_2]$. 
The case $m\geq 3$ is more subtle, since the diagonal $\Delta_m\subset \Sym^mC$
lies on the boundary of the effective cone of $\Sym^m C$, and even its numerical properties
seem rather mysterious.\\
In a deeper vein, the next proposition shows that
pairs of Kodaira fibrations admit
common refinements, meaning that they can dominated simultaneously by a third 
Kodaira fibration.
\begin{proposition}\label{fibered}
Let $q_1:X_1\to C_1$ and $q_2:X_2\to C_2$ be smooth fibrations in curves.
 Then there exists a smooth curve $C_0$ with surjective morphisms
$t_i: C_0\to C_i$, a fibration in curves $X\to C_0$ and finite $C_0$-morphisms $X\to t_i^*X_i$, $i=1,2$.
\end{proposition}
\begin{proof}
We first fix some notation: for a given curve $D$, denote by $\mathscr M_g(D)$ the substack of $\mathscr M_g$ parametrizing smooth
genus $g$ curves
$C$ that admit a surjective morphism $C\to D$. Similarly, for a given pair of curves $D_1,D_2$, denote by 
$\mathscr M_g(D_1,D_2)$ the substack of $\mathscr M_g$ parametrizing smooth
genus $g$ curves
$C$ that admit surjective morphisms $C\to D_1$ and $C\to D_2$.
The above admits a relative analogue, in that
if $\mathscr D_1\to B$ and $\mathscr D_2\to B$ are two families of curves, 
we have an induced fibration $\mathscr M_g(\mathscr D_1,\mathscr D_2/B)\to B$, whose fiber over $b\in B$
is $\mathscr M_g(\mathscr {D}_{1,b},\mathscr D_{2,b} )$.
A trick by Kodaira, \cite{20} 2.33, shows that for every $n$, there exists $g$ such that
$\mathscr M_g(D)$ contains complete subvarieties of dimension $n$, and therefore the same holds for $\mathscr M_g(D_1,D_2)$.\\
We are now ready to offer a proof of the proposition: let $C$ be a curve with surjective morphisms $t_i:C\to C_i$ and replace $q_i$ by $p_i:Y_i:=t_i^*X_i\to C$. 
Consider the fibration $\pi: \mathscr M_g(Y_1, Y_2/C)\to C$. By the above remarks, for $g$ big enough there
 exists a surface $Z\subset \mathscr M_g(Y_1, Y_2/C)$
such that $\pi_{|Z}:Z\to C$ is surjective and the fiber $Z\cap \pi^{-1}(c)$ is a complete curve for generic $c\in C$.
Such morphism $\pi_{|Z}$ clearly admits a multisection $C_0$ - for example by taking
 an ample divisor in $\overline Z$, the closure of $Z$
in $\overline{\mathscr M_g}(Y_1, Y_2/C)$, missing the isolated boundary points - and such multisection $C_0$ induces, by definition,
a smooth fibration in curves $X\to C_0$ dominating $Y_1/C$ and $Y_2/C$.
\end{proof}
\end{section}

\begin{section}{Smooth foliations on surfaces of general type}
\label{four}
Motivated by the study of surfaces that carry smooth fibrations,
we dedicate this final section to surfaces of general type carrying smooth foliations. In particular
we look for restrictions the existence of a smooth foliation imposes on the ambient surface, 
and then try to understand what they might possibly look like.
Recall that a smooth foliation $\mathscr F$ on an algebraic surface $X$ defines an exact sequence of vector bundles
$$0\to T\mathscr F\to TX\to N\mathscr F\to 0$$
The next fact recollects some well known numerical properties of $K\mathscr F:=T\mathscr F^{\vee}$.
\begin{fact}\label{numerology}
\begin{itemize}
\item []
 \item [(i)]  $c_2(X)=K\mathscr F\cdot (K_X-K\mathscr F)$.
 \item [(ii)]  $(K_X-K\mathscr F)^2=0$.
 \item [(iii)]  $c_2(X)=K_X^2-K\mathscr F\cdot K_X $.
 \item [(iv)] If $X$ is of general type then $K\mathscr F$ is pseudoeffective, and it is big unless $\mathscr F$ is either
 an isotrivial fibration, or a Hilbert modular foliation.

\end{itemize}

\end{fact}
\begin{proof}
\begin{itemize}
\item[]
 \item [(i)]  This follows by taking Chern numbers in the defining exact sequence of $T\mathscr F$, plus $c_2(X)>0$.
 \item [(ii)]  This is the Baum-Bott index formula, $N\mathscr F^2=0$.
 \item [(iii)]  This is a formal consequence of (i) and (ii).
 \item [(iv)] The pseudoeffectivity of $K\mathscr F$ follows from the Main Theorem of \cite{10}, since $X$ has general type.
By the birational classification of foliations on surfaces, \cite{13},
if a foliation on a surface of general type is such that $K\mathscr F$ is not big, then
the foliation is an isotrivial or Hilbert modular.
\end{itemize}

\end{proof}
Therefore, granted a decent amount of understanding of isotrivial and Hilbert modular
surfaces, we concentrate on the strong topological and algebraic restrictions, for
a surface of general type to carry a smooth foliation.

\begin{proposition}\label{foliation}
Let $X$ a smooth surface of general type, and $\mathscr F$ an everywhere smooth foliation on $X$
with $K\mathscr F$ big.
Then we have
\begin{itemize}
 \item  $\rk \Pic(X)\geq 2$.
 \item  $X$ has positive topological index, i.e. $c_1(X)^2> 2c_2(X)$.
\end{itemize}

\end{proposition}

\begin{proof}

Let $a:=\frac{K_X\cdot K\mathscr F}{K_X^2}$, and $R:=aK_X - K\mathscr F$. Since $R\cdot K_X=0$
we deduce, by Hodge index theorem, that $R^2<0$ unless $R$ is numerically trivial.
\begin{claim}\label{R}
$R\neq 0$
\end{claim}
\begin{proof}
We have $$0=R\cdot K_X= -K\mathscr F\cdot K_X + K_X^2 + (a-1) K_X^2=c_2(X)+(a-1)K_X^2$$ so then
$$c_2(X)=(1-a)K_X^2$$
If $R=0$ then $aK_X=K\mathscr F$, and $(K_X-K\mathscr F)^2=0$ implies $a=1$, from which $c_2(X)=0$,
impossible since $X$ has general type.
\end{proof}
From which we deduce that $\rk \Pic(X)\geq 2$.\\ 
Let $P\subseteq NS(X)$ be the plane spanned by $K_X$ and $R$.
We know that $$\Amp(X)=\{D: D^2>0,D\cdot K_X>0\}$$ 
so then Fact \ref{numerology} (i) and (ii)
imply that $K_X-K\mathscr F=(1-a)K_X+R$ lies on the boundary of $\Amp(X)\cap P$. 
It follows that $\Amp(X)\cap P$
is bounded by the rays $$(1-a)K_X + R,\ (1-a)K_X - R$$
Since $$K\mathscr F=aK_X - R\in \Amp(X)\cap P$$ by assumption, we have 
$a > 1-a$, or $a> 1/2$.
The identity $$c_2(X)=(1-a)K_X^2$$ derived in the proof of Claim \ref{R}, 
concludes the proof.
\end{proof}

More interestingly, Brunella \cite{9} has initiated the uniformization theory
on smoothly foliated surfaces $(X,\mathscr F)$ of general type: the universal cover $\tilde{X}$ of
such a surface, admits
a smooth holomorphic fibration $p_{\mathscr F}:\tilde{X}\to \Delta$ onto the unit disk, with disks as fibers,
such that $\mathscr F$ becomes tangent to $p_{\mathscr F}$ on $\tilde{X}$.
In particular, $X$ is a $K(\Gamma,1)$, for $\Gamma:=\pi_1(X)$.
With the aim of better understanding the geometry of $X$, observe that
the fibration $p_{\mathscr F}$ is naturally $\Gamma$-equivariant, and hence
there is a natural representation $\rho:\Gamma\to PSL_2(\mathbf R)$.
Corlette and Simpson, in \cite{14}, have given a complete classification
of what such a $\rho$ can be, and deduced a beautiful dichotomoy
for its geometric origin. We can summarize all of this in:
\begin{theorem}[Brunella, Corlette \& Simpson] \label{BCS}
 Let $(X,\mathscr F)$ be a smoothly foliated surface of general type.
 Then at least one of the following happens:
 \begin{itemize}
  \item $X$ admits a smooth fibration $p:X\to C$ onto a smooth curve $C$
  and $\mathscr F$ is tangent to $p$;
  \item $\rho$ is rigid and integral, there exists a 
  quasiprojective polydisk quotient, $Y$, and a natural morphism $X\to Y$ such that
  $\mathscr F$ is induced by one of tautological codimension one foliations on $Y$.
 \end{itemize}

\end{theorem}
The only issue to be settled here, is the relation between the dimensions of $X$
and $Y$. It seems plausible, thinking about the wild behaviour of Hilbert
modular foliations,
that indeed the dimensions must be the same, and $X$ itself is a bidisk quotient.\\
Let us remark that bounded symmetric domains can be used in our problem of
finding dominant classes as follows:
\begin{proposition}\label{balls}
 Let $X$ be a smooth algebraic surface. Then there exists a finite morphism
 $Y\to X$, such that an \'etale cover $Y'$ of $Y$ is a Stein submanifold
 of a $3$-dimensional ball.
\end{proposition}
\begin{proof}
 Let $Z$ be a ball quotient, consider projections $X\to \mathbf P^2$,
 $Z\to \mathbf P^2$, and let $Y:=X\times_{\mathbf P^2} Z$. For sufficiently generic 
 projections, the branch loci in $\mathbf P^2$ are transverse, hence $Y$ is smooth.
 Let $p:\mathbf B^2\to Z$ be the universal cover, then $Y':=p^*Y$ is a complex
 manifold which is naturally a finite ramified cover of the ball $\mathbf B^2$,
 and as such embeds into a product $\mathbf B^2\times \Delta$.
\end{proof}

We wish to conclude by turning our attention to the corresponding problem in
characteristic $p>0$, of understanding smooth foliations on algebraic
surfaces. Perhaps not too surprisingly, the situation is drastically 
different from the complex case, and it turns out that it is extremely simple
to construct a covering of a given surface that carries a smooth foliation.
Before proceeding to our main result, let us recall a key difference between 
foliations in characteristic $0$ and $p$, that lies in the notion of
integrability: the celebrated Frobenius integrability theorem, as well known
for integrable distributions over $\mathbf C$, fails over fields of positive
characteristic: integrable distributions need not have a formal first integral.
Observe, indeed, that the notion of leaf is rather badly behaved, and the closest
we can get to formal integrability is $p$-closedness - that is, the algebra
generated by the vector fields defining our foliation is closed under $p$-powers.
This implies that the kernel of such algebra of differential operators
defines a factorization of the Frobenius morphism on the ambient variety. And
a modicum of thought shows that this is the best integrability condition one can hope
for. In practice, probably the quickest way of appreciating the role of
$p$-closedness in remedying the failure of formal integrability is:
\begin{fact}\cite [II.1.6]{13}\label{Mac}
Let $A$ be a complete regular local ring over an algebraically
closed field of characteristic $p$, and let $\partial$ be a non-singular
derivation of $A$. Then there exists a regular system $x,y_1,...,y_n$
of parameters such that 
$$\partial= \frac{\partial}{\partial x}+ x^{p-1}\sum_{i=1}^n f_i(x^p,y)\frac{\partial}{\partial y_i}$$
Moreover the ideal defining the vanishing of $\partial\wedge \partial^p$ is generated
by $f_1,...,f_n$. In particular $\partial$ is $p$-closed iff $\partial=\frac{\partial}{\partial x}$,
i.e. a smooth foliation by curves is $p$-closed iff it is formally integrable.
\end{fact}

\begin{subsection}{Proof of Theorem \ref{charpfoliations}.}

Let $p:Y_0\to \mathbf P^1_k$ be a general Lefschetz pencil in $X$, whose general member
is smooth, and whose singular members have exactly one node.
Observe that $Y_0$ is obtained by blowing up $X$ along the base points of the pencil.
Let $S\subset \mathbf P^1_k$ parametrize the singular fibers,
and for each $s\in S$ let $y_s\in Y_s$ be the node in the fiber.
Let $H$ be a very ample divisor, and consider the sheaf
$\mathscr O(pH)\otimes I_{\cup_s y_s}$ whose local sections are
those of $\mathscr O(pH)$ vanishing along $\cup_s y_s$.
Its generic global section defines a curve, $D$, which is smooth, transverse to
the branches of $Y_s$ at $y_s$, and has simple tangencies with the fibers of
$p$ outside their nodes. The following claim will conclude:\\
\begin{claim}
 Let 
 $$r_D:Y:=Y_0(\sqrt[p]{D})\to Y_0$$ denote the inseparable cyclic $p$-cover, branched 
 along $D$. If $\mathscr F$ denotes the foliation defined
 by $p$, then the saturation of $r_D^*\mathscr F$
 is smooth and $p$-closed.
\end{claim}
\begin{proof}
First we deal with smoothness.
We need to worry about what happens at:
\begin{itemize}
 \item[(i)] The nodes in the fibers of $p$, and
 \item[(ii)] The simple tangencies between $D$ and the smooth locus of the fibers of $p$.
\end{itemize}
Let us consider (i). In the local ring of $Y_0$ completed in a node in a singular fiber, the branches of
the singular fiber give us local coordinates $x,y$, while
$D$ is defined by the vanishing of a third local function, $z$.
$\mathscr F$ is defined by the vanishing of the form $d(xy)$, and
our assumptions on $D$ imply that there exists a formal function
$G$ such that $y=G(x,z)$ holds.
The cyclic cover $r_D$ is defined by $r_D^*z=z^p$, hence
$$r^*d(xy)=d(x\cdot G(x,z^p))=(G+x\frac{\partial G}{\partial x}) dx$$
and therefore, upon saturation, $r_D^*\mathscr F$ is smooth
around the pre-image, under $r_D$, of nodes in the fibers of $p$.\\
Let us deal with (ii).
In the local ring of $Y_0$ completed in a point of simple tangency, we can pick
local coordinates $x,y$ such that
our fibration is defined by the vanishing of the form $dy$.
If the vanishing of $z$ defines $D$ in such coordinates, simple tangency means that
there exists a formal function $g$ such that $z=y-g(x)$
and $\ord_x g(x)=2$.\\
\textit{In particular, $\frac{dg(x)}{dx}$ is not identically zero, since $p>2$}.\\
The cyclic cover $r_D$ is defined by
$$r_D^*z=z^p,\ r_D^*x=x,\ r_D^*y=z^p+g(x)$$
and the pullback foliation is then
$$r_D^*dy=d(z^p+g(x))=\frac{dg(x)}{dx}dx$$
which is again smooth upon saturation.\\
Observe, as an output of our local computations, that
the saturation of $r_D^*\mathscr F$ is not only smooth, but
also formally integrable. By Fact \ref{Mac}, it is $p$-closed.
\end{proof}
\end{subsection}
\end{section}


\begin{thebibliography}{ELMNPM}
\bibitem[A74]{6} M. Artin, \textit{Algebraic construction of Brieskorn's resolutions}, J. Alg.,
Vol. 29, Issue 2, May 1974, pp. 330-348, \url{http://www.sciencedirect.com/science/article/pii/0021869374901021}

\bibitem[BCF15]{12} I. Bauer, F. Catanese, D. Frapporti, \textit{The fundamental group and torsion group of Beauville
surfaces}, Beauville surfaces and groups, 1–14, Springer Proc. Math. Stat., 123, Springer, Cham, 2015. 14Fxx (14J25),
\url{http://arxiv.org/abs/1402.2109}

\bibitem[BH00]{1} F. Bogomolov, D. Husemoller, \textit{Geometric properties of curves defined over number fields}, MPI preprint
, 2000-1, \url{http://www.mpim-bonn.mpg.de/preprints}

\bibitem[BM16]{10} F. Bogomolov, M. McQuillan, \textit{Rational curves on foliated varieties},
Foliation Theory in Algebraic Geometry. Springer International Publishing, 2016. 21-51,
\url{http://link.springer.com/content/pdf/10.100
7%2F978-3-319-24460-0_2.pdf}

\bibitem[BT02]{4} F. Bogomolov, Y. Tschinkel, \textit{Unramified correspondences}, Algebraic number theory and algebraic geometry, Contemp.
Math., vol. 300, Amer. Math. Soc., Providence, RI, 2002,
pp. 17-25, \url{http://arxiv.org/abs/math/0202223}

\bibitem[BT02']{8} F. Bogomolov, Y. Tschinkel, \textit{On curve correspondences}, Communications in Arithmetic
 Fundamental Groups and Galois Theory, RIMS, pp.157-166, 2002, \url{http://www.math.nyu.edu/~tschinke/papers/yuri/02genram/genram.pdf}
 
\bibitem[BT05]{5} F. Bogomolov, Y. Tschinkel, \textit{Couniformization of curves over number fields},  
Geometric methods in algebra and number theory, pp. 43-57, Progress in Math. 235, Birkh�user, 2004
, \url{http://www.math.nyu.edu/~tschinke/papers/yuri/04covers/cover4.pdf}

\bibitem[Br97]{9} M. Brunella, \textit{Feuilletages holomorphes sur les surfaces complexes compactes},
Annales scientifiques de l'\'Ecole Normale Sup\'erieure (1997)
Volume: 30, Issue: 5, page 569-594
\url{http://archive.numdam.org/ARCHIVE/ASENS/ASENS_1997_4_30_5/ASENS_1997_4_30_5_569_0/ASENS_1997_4_30_5_569_0.pdf}

\bibitem[CC98]{11} L. Chiantini, C. Ciliberto, \textit{On the Severi varieties of surfaces in $\mathbf P^3$},
\url{http://arxiv.org/abs/math/9802009}

\bibitem[CS08]{14} K. Corlette, C. Simpson, \textit{ On the classification of rank-two representations
of quasiprojective fundamental groups}, Compos. Math. 144
(2008), 1271–1331, \url{https://arxiv.org/pdf/math/0702287.pdf}

\bibitem[SGA7.2]{3} P. Deligne, N. Katz, \textit{S\'eminaire de G\'eom\'etrie Alg\'ebrique du Bois Marie -
 1967-69 - Groupes de monodromie en g\'eom\'etrie alg\'ebrique
 - (SGA 7) - vol. 2}. LNM 340. Berlin, New York: Springer-Verlag. pp. x+438, 1973,
 \url{https://publications.ias.edu/sites/default/files/Number12.pdf}
 
\bibitem[dJ96]{2} A.J. de Jong, \textit{Smoothness, semi-stability and alterations},  Publications Math\'ematiques de l'IH\'ES 83,
1996, pp. 51-93, 
\url{http://www.math.uiuc.edu/K-theory/0081}

\bibitem[SGA1]{21} A. Grothendieck, \textit{S\'eminaire de G\'eom\'etrie Alg\'ebrique du Bois Marie - 1960-61 -
Rev\^etements \'etales et groupe fondamental - (SGA 1)}. LNM 224. Springer-Verlag. pp. xviii+327, 1971,
\url{https://arxiv.org/abs/math/0206203}

\bibitem[HM98]{20} J. Harris, I. Morrison, \textit{Moduli of curves}, 
Graduate Texts in Mathematics, 187. Springer-Verlag, New York, 1998. xiv+366 pp,
\url{http://link.springer.com/content/pdf/10.1007%2Fb98867.pdf}

\bibitem[Ko67]{7} K. Kodaira, \textit{A certain type of irregular algebraic surfaces} Journal d'Analyse Mathematique, vol. 19 (1967), pp. 207-215

\bibitem[McQ08]{13} M. McQuillan, \textit{Canonical models of foliations},  Pure Appl. Math. Q. 4 (2008), no. 3, part 2, 877–1012,
\url{http://www.mat.uniroma2.it/~mcquilla/files/canmod.pdf}

\end{thebibliography}
\end{document}